\def\draft{n}
\documentclass{amsart}
\usepackage[headings]{fullpage}
\usepackage{amssymb,epic,eepic,epsfig}
\usepackage[all]{xy}
\usepackage{graphicx}
\usepackage{texdraw}
\usepackage{url}
\usepackage[bookmarks=true,%
    colorlinks=true,%
    linkcolor=blue,%
    citecolor=blue,%
    filecolor=blue,%
    menucolor=blue,%
    urlcolor=blue,%
    breaklinks=true]{hyperref}

\newtheorem{theorem}{Theorem}[section]
\newtheorem{proposition}{Proposition}[section]
\theoremstyle{definition}
\newtheorem{lemma}[proposition]{Lemma}

\newtheorem{remark}[proposition]{Remark}

\newtheorem{conjecture}[proposition]{Conjecture}

\newtheorem{problem}[proposition]{Problem}
\newtheorem{example}[proposition]{Example}

\newtheorem{claim}[proposition]{Claim}

\def\printname#1{
        \if\draft y
                \smash{\makebox[0pt]{\hspace{-0.5in}
                        \raisebox{8pt}{\tt\tiny #1}}}
        \fi
}

\newlength{\standardunitlength}
\catcode`\@=11
\long\def\@makecaption#1#2{%
     \vskip 10pt

\setbox\@tempboxa\hbox{
       \small\sf{\bfcaptionfont #1. }\ignorespaces #2}%
     \ifdim \wd\@tempboxa >\captionwidth {%
         \rightskip=\@captionmargin\leftskip=\@captionmargin
         \unhbox\@tempboxa\par}%
       \else
         \hbox to\hsize{\hfil\box\@tempboxa\hfil}%
     \fi}
\font\bfcaptionfont=cmssbx10 scaled \magstephalf
\newdimen\@captionmargin\@captionmargin=2\parindent
\newdimen\captionwidth\captionwidth=\hsize
\catcode`\@=12

\newcommand{\tr}{\operatorname{tr}}

\def\lbl#1{\label{#1}\printname{#1}}

\def\BZ{\mathbb Z}
\def\BP{\mathbb P}
\def\BQ{\mathbb Q}
\def\BR{\mathbb R}
\def\BC{\mathbb C}
\def\BH{\mathbb H}

\def\BK{\mathbb K}

\def\calT{\mathcal T}
\def\calG{\mathcal G}

\def\a{\alpha}

\def\l{\lambda}
\def\Ga{\Gamma}

\def\ga{\gamma}

\def\e{\epsilon}
\def\Ga{\Gamma}

\newcommand{\ZZ}{\mathbb{Z}}
\newcommand{\CC}{\mathbb{C}}

\newcommand{\sll}{\mathfrak{sl}_2(\CC)}

\def\longto{\longrightarrow}

\def\SL{\mathrm{SL}}
\def\PSL{\mathrm{PSL}}
\def\pt{\partial}

\def\CS{\mathrm{CS}}
\def\Vol{\mathrm{Vol}}

\def\m{\mu}

\def\CS{\mathrm{CS}}
\def\Hom{\mathrm{Hom}}

\def\SL{\mathrm{SL}}

\begin{document}


\title[Rationality of the $\SL(2,\BC)$-Reidemeister torsion in dimension 3]{
Rationality of the $\SL(2,\BC)$-Reidemeister torsion in dimension 3}
\author{Jerome Dubois}
\address{Institut de Math\'ematiques de Jussieu \\
Universit\'e Paris Diderot--Paris 7 \\
UFR de Math\'ematiques\\
Case 7012, B\^atiment Chevaleret\\
2, place Jussieu\\
75205 Paris Cedex 13
FRANCE \\ \newline
        {\tt \url{http://www.institut.math.jussieu.fr/$\sim$dubois/ }}} 
\email{dubois@math.jussieu.fr}
\author{Stavros Garoufalidis}
\address{School of Mathematics \\
         Georgia Institute of Technology \\
         Atlanta, GA 30332-0160, USA \newline 
         {\tt \url{http://www.math.gatech.edu/~stavros }}}
\email{stavros@math.gatech.edu}
\thanks{S.G. was supported in part by NSF. \\
\newline
1991 {\em Mathematics Classification.} Primary 57N10. Secondary 57M25.
\newline
{\em Key words and phrases: knots, $A$-polynomial, Reidemeister torsion, 
volume, character variety, 3-manifolds, hyperbolic geometry, 
invariant trace field.
}
}

\date{February 10, 2010 }


\begin{abstract}
If $M$ is a finite volume complete hyperbolic 3-manifold with one cusp and no
$2$-torsion,
the geometric component $X_M$ of its $\SL(2,\BC)$-character
variety is an affine complex curve, which is smooth at the discrete faithful
representation $\rho_0$. Porti defined a non-abelian Reidemeister torsion in a 
neighborhood of $\rho_0$ in $X_M$ and observed that it is an analytic map,
which is the germ of a unique rational function on $X_M$. In the present
paper we prove that (a) the torsion of a representation lies
in at most quadratic extension of the invariant trace field of the
representation, and (b) the existence of a polynomial relation of the 
torsion of a representation and the trace of the meridian or the 
longitude. We postulate that the coefficients of the $1/N^k$-asymptotics
of the Parametrized Volume Conjecture for $M$ are elements of the field
of rational functions on $X_M$.  
\end{abstract}

\maketitle

\tableofcontents

\section{Introduction}
\lbl{sec.intro}

\subsection{The volume  of an $\SL(2,\BC)$-representation and the 
$A$-polynomial}
\lbl{sub.volume}

A well-known numerical invariant of a 3-dimensional finite volume 
hyperbolic manifold $M$ with a cusp is its {\em volume}, a positive real 
number. A complete invariant of the hyperbolic structure of $M$ is
a discrete faithful representation of $\pi_1(M)$ into $\PSL(2,\BC)$ 
(well-defined
up to conjugation) which is also a topological invariant, as follows
from Mostow rigidity Theorem.
Every $\PSL(2,\BC)$-representation $\rho$ of $\pi_1(M)$ has a real-valued 
volume $\Vol(\rho)$; see \cite[Ch.2]{Dn} and also \cite{F,FK}.
When a representation varies in a 1-parameter family $\rho_t$, 
the variation of the volume $\frac{d}{dt}\Vol(\rho_t)$ depends only
on the restriction of $\rho_t$ to the boundary torus $\pt M$. This is
a general principle of Atiyah-Patodi-Singer, and in our special case it
also follows from Schalfi's formula. This raises the question:
which $\PSL(2,\BC)$-representations of $\pt M$ extend to a representation of
$M$? The answer is given by an algebraic condition between the
eigenvalues of a meridian and longitude of $\pt M$. This condition
is the vanishing of the so-called \emph{$A$-polynomial} of $M$; see 
\cite{CCGLS}.
The $A$-polynomial of $M$ encodes important informations about
\begin{itemize}
\item[(a)]
the hyperbolic
geometry of $M$, and determines the variation of the volume of the hyperbolic
structure of $M$.
\item[(b)]
the topology of $M$ and more precisely about the
slopes of incompressible surfaces in the knot complement, as follows from 
Culler-Shalen theory; see \cite{CCGLS}.
\end{itemize}
More recently, the $A$-polynomial (or rather, its extension that 
includes the images of all components of the character variety)
is conjecturally linked in two different ways to a {\em quantum knot 
invariant}, namely the {\em colored Jones polynomials} of 
a knot in 3-space (for a definition of the latter, which we will not
use in the present paper, see \cite{Tu} and \cite{GL1}):

\begin{itemize}
\item[(a)]
There is an $A_q$-polynomial in two $q$-commuting variables which 
encodes a minimal order linear $q$-difference equation for the sequence
of colored Jones polynomials; see \cite{GL1}. The AJ Conjecture of \cite{Ga1}
states that when $q=1$, the $A_q$-polynomial coincides with the $A$-polynomial.
\item[(b)]
There is a parametrized version of the Volume Conjecture which links
the variation of the limit in the Volume Conjecture to the $A$-polynomial;
see \cite{GM,GL2}.
\end{itemize}

Aside from conjectures, the following result of \cite{DG} and \cite{BZ}
(based on foundational work of Kronheimer-Mrowka)
shows that the $A$-polynomial detects the unknot.

\begin{theorem}\cite{BZ,DG}
\lbl{thm.nontrivialA}
The $A$-polynomial of a nontrivial knot in 3-space is nontrivial.
\end{theorem}

\subsection{The $\SL(2,\BC)$-character variety of $M$ and its 
field of rational functions}
\lbl{sub.rational}

For historical reasons that simplify the linear algebra, it is useful
to consider $\SL(2,\BC)$ (rather than $\PSL(2,\BC)$)-representations of
$\pi_1(M)$. In the rest of the paper, $M$ will denote a finite volume
hyperbolic 3-manifold with one cusp, such that the homology of $M$
contains no $2$-torsion.
In this case, the discrete faithful representation of $M$ lifts to a 
$\SL(2,\BC)$-representation $\rho_0\colon \pi_1(M) \to \SL(2,\BC)$; 
see \cite{Cu}.
To understand how the $\SL(2,\BC)$-representation $\rho_0$ of $\pi_1(M)$ varies,
we consider the unique component $X_M$ of the 
$\SL(2,\BC)$-{\em character variety} of $M$ that contains $\rho_0$.
It is well-known that $X_M$ is an affine curve defined over $\BQ$
and that $\rho_0$ is a smooth point of $X_M$; see \cite{CCGLS}. Moreover,
the coordinate ring $\BQ[X_M]$ is generated by $\tr_{\ga}$ for all $\ga \in
\pi_1 M$, where $\tr_{\ga}$ is the so called \emph{trace-function} defined by:
\begin{equation}
\lbl{eq.trgamma}
\tr_{\ga}:X_M \longto \BC, \qquad \tr_{\ga}(\rho)=\tr(\rho(\ga)).
\end{equation}
Here $\tr(A) = \sum_i a_{ii}$ denotes the trace of a square matrix 
$A = \left( a_{ij}\right)$. Let $\BQ(X_M)$
denote the field of rational functions of $X_M$.
For a detailed discussion on character varieties, the reader may consult
Shalen's survey~\cite{Sh} and also ~\cite[Sec.10]{BDR-V}
and \cite{CCGLS,Go}.

\subsection{The Reidemeister torsion of an $\SL(2,\BC)$-representation}
\lbl{sub.taupoly}

Another important numerical invariant of a representation of a manifold is its 
{\em Reidemeister torsion}, which comes in several
combinatorial or analytic flavors, see Milnor's survey \cite{Mi} or Turaev's 
monograph \cite{TM} for details.
Combinatorially,   the Reidemeister torsion is defined in terms of ratios of
determinants of matrices assigned to based, acyclic complexes, which themselves
are associated with a cell decomposition of a manifold and an acyclic 
representation. One can define torsion for all (not
necessarily acyclic) representations of a manifold as an element of a 
top  exterior power of a twisted (co)homology group, and one can obtain
a complex number after choosing a basis for the twisted (co)homology.
Porti~\cite{Po} defined a Reidemeister torsion for the adjoint 
representation associated to an $\SL(2,\BC)$-representation 
$\rho$ of $\pi_1(M)$ when $\rho$ is in a {\em neighborhood} $U$
of $\rho_0 \in X_M$.  Such representations are not acyclic and a basis
for the twisted homology (and thus the torsion) 
depends on an {\em admissible curve} $\ga$, i.e., 
a simple closed curve $\ga$ in $\pt M$ which is not nullhomologous in 
$\pt M$ (see Porti's monograph~\cite[Chap. 3]{Po} for details). 
Thus, the {\em non--abelian Reidemeister torsion} is a map:
\begin{equation}
\lbl{eq.torsion}
\tau_{\ga}\colon U \longto \BC.
\end{equation}
Moreover, Porti~\cite{Po} observed that $\tau_\gamma$ is an analytic map.
In addition Porti obtain the following result.

\begin{theorem}
\lbl{thm.1}\cite[Thm.4.1]{Po}
For every admissible curve $\ga$, the non-abelian Reidemeister torsion 
$\tau_{\gamma}: U \longto \BC$ is the germ of a unique element
of $\BQ(X_M)$, which is regular at $\rho_0$.
\end{theorem}
In Section \ref{sub.thm1} we will give an independent proof of Theorem
\ref{thm.1}, which we need for the main results of our paper. 
To phrase our results, recall that the {\em trace field} $\BQ(\rho)$ of
an   $\SL(2,\BC)$-representation $\rho$ of $M$ is the field $\BQ(\tr_g(\rho)|
g \in \pi_1(M))$. For an admissible curve $\ga$, let $\{e_{\ga}(\rho),
e_{\ga}(\rho)^{-1}\}$
denote the eigenvalues of $\rho(\ga)$. Observe that the field 
$\BQ(\rho)(e_{\ga}(\rho))$ is at most a quadratic extension of the 
trace field of $\rho$. Our next theorem uses the notion
of a {\em generic representation}, defined in Section \ref{sec.model}. Note
that this is a Zariski open condition, and that the discrete faithful
representation is generic   (\emph{regular} in the language of Porti's work).

\begin{theorem}
\lbl{thm.11}
For every admissible curve $\ga$ and every generic representation $\rho$,
$\tau_{\ga}(\rho)$ lies in the field $\BQ(\rho)(\e_{\ga}(\rho))$. 
In particular, $\tau_{\ga}(\rho_0)$ lies in the trace field of $M$.
\end{theorem}
Note that since the homology of $M$ has no $2$-torsion, the trace
field of $M$ coincides with its invariant trace field; 
see \cite[Thm.2.2]{NR}.
Our next theorem shows that $\tau_{\ga}$ is an {\em algebraic function} of 
$\tr_{\ga}$. This follows easily from the fact that $\tau_{\ga}$
and $\tr_{\ga}$ are rational functions on $X_M$ and that $\BQ(X_M)$
has transcendence degree $1$, since $X_M$ is an affine
curve defined over $\BQ$.  

\begin{theorem}
\lbl{thm.2}
For every admissible curve $\ga$, there exists a polynomial $T_{\ga}(\tau,y) 
\in \BZ[\tau,y]$, called the $T_\ga$-polynomial, so that 
\begin{equation}
T_{\ga}(\tau_{\ga},\tr_{\ga})=0.
\end{equation}
\end{theorem}

Let us make some remarks regarding Theorems \ref{thm.1} and \ref{thm.2}.

\begin{remark}
\lbl{rem.TA}
The dependence of the torsion function $\tau_{\ga}$
on $\gamma$ is determined by the $A$-polynomial; see Equation \eqref{eq.tauml}. 
Thus, $T_{\ga}$ is determined by $T_{\mu}$ and the $A$-polynomial of $M$. 
Moreover, if we let $\{e_\mu(\rho), e_\mu^{-1}(\rho)\}$ (resp. $\{e_\lambda(\rho), e_\lambda^{-1}(\rho)\}$) de the eigenvalues for the meridian $\mu$ (resp. longitude $\lambda$) at $\rho$, that is to say, if
\[
e_\mu(\rho) + e_\mu^{-1}(\rho) = \mathrm{tr}_\mu(\rho) \text{ and } e_\lambda(\rho) + e_\lambda^{-1}(\rho) = \mathrm{tr}_\lambda(\rho)
\]
then one has (see~\cite[Thm.4.1]{Po}):
\[
\tau_\lambda = \frac{e_\mu}{e_\lambda} \cdot \frac{\partial e_\lambda}{\partial e_\mu} \cdot \tau_\mu
\]
In particular, at the discrete faithful representation $\rho_0$,
we have:
\begin{equation}
\lbl{eq.comparecusp}
\tau_\lambda(\rho_0)=\mathfrak{c}\cdot  \tau_\mu (\rho_0)
\end{equation}
where $\mathfrak{c}$ is the {\em cusp-shape}. This holds since near $\rho_0$ we have
$A(1+t+O(t)^2,-1+ \mathfrak{c} \, t+O(t^2))=0$ where 
$A(M,L)$ is the $A$-polynomial.
\end{remark}

\begin{remark}
\lbl{rem.alg}
Theorem \ref{thm.1} is an instance of a well-recorded phenomenon: 
many classical and quantum invariants of knotted 3-dimensional objects 
are algebraic. For a detailed discussion regarding conjectures and facts,
see \cite{Ga2}. For a quick explanation of the algebricity in dimension 3,
see Section \ref{sub.explanation} below. 
\end{remark}

\subsection{Examples}
\lbl{sub.examples}

In this section, we illustrate Theorem \ref{thm.2} for the complement
of the figure eight knot $4_1$, and the complement of the $5_2$ knot.

\begin{example}
\lbl{ex.1}
Consider the complement $M$  of the figure eight knot $4_1$ with a 
meridian-longitude system $(\mu,\l)$. The non--abelian Reidemeister torsion 
(with respect to the longitude $\l$) on the character variety $X_M$ is given 
by (see~\cite{Po} or~\cite{Db}):
$$
\tau_\lambda = \sqrt{17 + 4 \tr_\lambda}.
$$
with the convention that we choose the positive square root near the
discrete faitfhul representation $\rho_0$ with $\tr_{\l}(\rho_0)=-2$   (see~\cite[Cor.2.4]{Ca}).
Thus $T_{\l}(\tau_{\l},\tr_{\l})=0$ where
$$
T_\lambda(x, y) = 17 + 4 y - x^2.
$$
Let $\tr_{\l}=e_\l+e_\l^{-1}$, $\tr_{\mu}=e_\m+e_\m^{-1}$. The vanishing of 
the $A$-polynomial for the figure eight knot gives us the following identity (see \cite{CCGLS}):
$$
A(\e_\l, e_\m) = -2 + (e_\m^4 + e_\m^{-4}) - (e_\m^2 + e_\m^{-2}) + (e_\l + e_\l).
$$
Thus, we obtain:
$$
\tr_\lambda = \tr_\mu^4 - 5\tr_\mu^2 + 2.
$$
For details, see~\cite{Po,DHY}. On the other hand, the torsion with respect
to the meridian is given by (see Equation \eqref{eq:changecurve}):
$$
\tau_\mu = 
\tau_\lambda \cdot \left( \frac{\tr_\lambda^2 - 4}{\tr_\mu^2 - 4}
\right)^{1/2} \cdot \frac{\partial \tr_\mu}{\partial \tr_\lambda}
= \frac{1}{2}\sqrt{(\tr_\mu^2 - 5)(\tr_\mu^2 - 1)}.
$$
Thus $T_{\mu}(\tau_\mu,\tr_{\mu})=0$ where
$$
T_\mu(\tau, z) = -5 + 6z^2 - z^4 + 4\tau^2.
$$
At the discrete faithful representation $\rho_0$, we have 
$\tr_{\l}(\rho_0)=-2$ (see~\cite[Cor.2.4]{Ca}) and $\tr_{\mu}(\rho_0)=\pm 2$ 
giving that
$$
\tau_\lambda(\rho_0)= 3, \qquad \tau_\mu(\rho_0)= \frac{i \sqrt{3}}{2}.
$$
On the other hand, the trace field of $4_1$ is $\BQ(x)$ where $x^2+3=0$.
This confirms Theorem \ref{thm.11} for the discrete faithful representation
$\rho_0$ of $4_1$. In addition, the cusp-shape of $4_1$ is $\mathfrak{c}=-2i\,\sqrt{3}$,
confirming Equation \eqref{eq.comparecusp}.
\end{example}

\begin{example}
\lbl{ex.2}
We will repeat the previous example for the twist knot $5_2$. 
The non--abelian Reidemeister torsion (with respect to the longitude $\l$) 
for $5_2$ is given by (see~\cite{DHY}):
$$
{\tau}_{\lambda}=(-10 \tr_\mu^2+21) +  \left(5 \tr_\mu^4 -27 \tr_\mu^2 + 35\right)u 
+\left(7 -5\tr_\mu^2\right)u^2,
$$
where $u$ satisfies the polynomial equation
$$
(2\tr_\mu^2 - 7)  - \left(\tr_\mu^4 - 7\tr_\mu^2 +14 \right)u 
+ \left(2\tr_\mu^2 - 7\right) u^2 - u^3 = 0.
$$
Eliminating $u$ from the above equations, it follows that 
$T_{\l}({\tau}_{\lambda},\tr_{\mu})=0$ where 
\begin{eqnarray*}
T_{\l}(x,y)&=&
x^3 
+ x^2 (35 - 26 y^2 + 5 y^4) 
+ x (294 - 280 y^2 + 83 y^4 - 10 y^6)
+ 343 + 196 y^2 - 126 y^4 + 20 y^6  
\end{eqnarray*}
We choose the branch of $u$ such that 
at the discrete faithful representation, $u_0$ satisfies the equation
$$
1 - 2 u_0 + u_0^2 - u_0^3=0, \qquad u_0=0.21508\dots - 1.30714 \dots \, i
$$
which coincides with the Riley polynomial of $5_2$; see \cite{MR}.
The invariant trace field of $5_2$ is the cubic subfield $\BQ(\a)$ of
the complex numbers given by: 
$$
\a^3-\a^2+1=0, \qquad \a=0.877439 \dots - 0.744862 \dots i
$$
and the cusp shape $\mathfrak{c}$ is given by:
$$
\mathfrak{c}=4 \a -6 = -2.49024\dots - 2.97945 \dots i
$$
which is related with the the root of the Riley polynomial by:
$$
u_0=\frac{4}{-\mathfrak{c}-2}
$$
The above equation agrees with \cite[Eqn.(3.9)]{DHY} up to the mirror
image of $5_2$. It follows that 
at the discrete faithful representation $\rho_0$, $\tau_{\l}(\rho_0)$ 
is the root of the equation
$$
\tau_{\l}(\rho_0)^3 + 11 \tau_{\l}(\rho_0)^2  - 138 \tau_{\l}(\rho_0) + 391 =0, 
\qquad \tau_{\l}(\rho_0) = 4.11623\dots  - 1.84036\dots \, i
$$
and in terms of the invariant trace field, is given by:
$$
\tau_{\l}(\rho_0)=-6 \a^2 + 13 \a -6
$$
Equation \eqref{eq.comparecusp} and the above discussion imply that:
$$
\tau_{\mu}(\rho_0)=\frac{\tau_{\l}(\rho_0)}{\mathfrak{c}}=
1 -\frac{3}{2} \a = -0.316158\dots  + 1.11729\dots \, i
$$
Notice that $-2 \tau_{\mu}(\rho_0)=3 \a-2$ is a prime of norm $-23$. In fact,
the invariant trace field $\BQ(\a)$ has discriminant $-23$ and $23$
ramifies as:
$$
-23=(3 \a-2)^2 (3 \a +1)
$$
where $3 \a-2$ and $3 \a +1$ are the primes above $23$.
The above discussion confirms Theorem \ref{thm.11} for the discrete 
faithful representation. 
\end{example}

\subsection{Problems}
\lbl{sub.problems}

In this section we list a few problems and future directions.

\begin{problem}
\lbl{prob.1}
Is the $T_{\l}$-polynomial of a hyperbolic knot nontrivial?
\end{problem}

\begin{remark}
\lbl{rem.puzzle}
The volume and the Reidemeister torsion appear as the {\em classical} and 
{\em semiclassical} limit in a parametrized version of the Volume 
Conjecture; see for example \cite{GM}. Physics arguments suggest that the 
non-commutative $A$-polynomial and the Reidemeister torsion is determined by 
the 
$A$-polynomial and the volume of the manifold alone. However, computations 
with twist knots suggest that the  $A$ and $T_{\l}$-polynomials seem to be 
independent from each other. Perhaps this discrepancy can be explained by
the difference between on-shell and off-shell physics computations.
\end{remark}

Let us now formulate a speculation regarding the 
{\em Parametrized Volume Conjecture} of Gukov-Murakami and Le-Garoufalidis; 
see \cite{GM,GL2}. If $K$ is a knot in $S^3$, let 
$J_{K,N}(q) \in \BQ[q^{\pm 1}]$ denote the quantum group invariant of $K$
colored by the $N$-dimensional irreducible representation of   
$\mathfrak{sl}_2(\CC)$,
and normalized to be $1$ at the unknot. For fixed $\a \in \BC$, the 
Parametrized Volume Conjecture studies the asymptotics of the sequence
$(J_{K,N}(e^{\a/N}))$ for $N=1,2,\dots$. For suitable $\a$ near $2 \pi i$, 
and for hyperbolic knots $K$, one expects an asymptotic expansion of the form
$$
J_{K,N}(e^{\a/N}) \sim e^{\frac{N \CS(\rho_{\a})}{2 \pi i}} N^{3/2}
c_0(\a)\left(1+\sum_{k=1}^\infty \frac{c_k(\a)}{N^k}\right)
$$
where $\rho_{\a} \in X_M$ denotes a representation near $\rho_0$ with 
$\tr_{\mu}(\rho_{\a})=e^{\a}+e^{-\a}$; see \cite{DGLZ,GL2}.

\begin{problem}
\lbl{prob.2}
For every $k$, and with suitable normalization, show that
$c_k(\a)$ are germs of unique elements of the field $\BQ(X_M)$.
\end{problem}

\begin{conjecture}
\lbl{conj.c0}
Show that 
\begin{equation}
\lbl{eq.c0}
c_0(0)=(2 \tau_{\mu}(\rho_0))^{-1/2}
\end{equation}
\end{conjecture}
H. Murakami has proven the above conjecture for the $4_1$ knot (see \cite{Mu}),
and unpublished computations of the second author and D. Zagier
have numerically verified the above conjecture for the $5_2$ and the
$(-2,3,7)$ pretzel knot. The details will appear in forthcoming work.

Our next problem concerns the extension of Theorem \ref{thm.1} to
simple complex Lie groups $G_{\BC}$, rather than $\SL(2,\BC)$.
Physics arguments regarding the 1-loop computation of 
{\em perturbative Chern-Simons theory} suggest 
that an extension of Theorem \ref{thm.1} to arbitrary complex simple groups 
$G_{\BC}$ is possible. It is reasonable to expect that
an extension of the non abelian Reidemeister torsion is possible (see for
example \cite{BH1,BH2}), and that Theorem \ref{thm.1} extends. 

\begin{problem}
\lbl{prob.3}
Extend Theorem \ref{thm.1} to arbitrary simple complex Lie groups $G_{\BC}$.
\end{problem}

\section{The character variety of hyperbolic 3-dimensional manifolds}
\lbl{sec.model}

\subsection{Four favors of the character variety, apr\`es Dunfield}
\lbl{sub.four}

The careful reader may observe that the volume function is defined
for $\PSL(2,\BC)$ representations of a 1-cusped hyperbolic manifold $M$, 
whereas the Reidemeister torsion is defined for $\SL(2,\BC)$-representations 
of $M$.
Our proof of Theorem \ref{thm.1} requires a new variant of a representation,
the so-called {\em augmented representation} that comes in two flavors:
the $\PSL(2,\BC)$ and the $\SL(2,\BC)$ one. For an excellent discussion, 
we refer the reader to \cite[Sec.2-3]{Dn} and \cite[Sec.10]{BDR-V}.
Much of the results of this section the second author 
learnt from N. Dunfield, whom we
thank for his guidance. Naturally, we are responsible for any comprehension
errors.

Let us define the four versions of the character variety of $M$. Let 
$R(M,\SL(2,\BC))$ denote the set of all 
homomorphisms of $\pi_1(M)$ into $\SL(2,\BC)$ 
and let   $X_{M,\SL(2,\BC)}$ be the set of \emph{characters} of $\pi_1(M)$ into $\SL(2, \BC)$ --- which is in a sense the  
algebrico-geometric quotient $R(M,\SL(2,\BC))/\SL(2,\BC)$, where $\SL(2,\BC)$ acts by 
conjugation (see~\cite{Sh}). The \emph{character} $\chi_\rho\colon \pi_1(M) \to \BC$ associated to the representation $\rho$ is defined by $\chi_\rho(g) = \mathrm{tr}(\rho(g))$, for all $g \in \pi_1(M)$. For irreducible representations, two representations are conjugate (in $\SL(2, \BC)$) if, and only if, they have the same character (see~\cite{CCGLS} or~\cite{Sh}). It is easy to see that 
$R(M,\SL(2,\BC))$ and $X_{M,\SL(2,\BC)} $ are affine
varieties defined over $\BQ$. 

Let $\overline{R}(M,\SL(2,\BC))$ denote
the subvariety of $R(M,\SL(2,\BC)) \times P^1(\BC)$ consisting of pairs 
$(\rho,z)$ where $z$ is a fixed point of $\rho(\pi_1(\pt M))$. Let 
$\overline{X}_{M,\SL(2,\BC)}$ denote the algebro-geometric quotient of 
$\overline{R}(M,\SL(2,\BC))$ under the diagonal action of $\SL(2,\BC)$
by conjugation and M\"obius transformations respectively. 
We will call elements $(\rho,z) \in \overline{R}(M,\SL(2,\BC))$ 
{\em augmented representations}. Their images in the augmented character 
variety $\overline{X}(M,\SL(2,\BC))$  will be called 
{\em augmented characters} and will be denoted by square
brackets $[(\rho,z)]$. Likewise, replacing $\SL(2,\BC)$ by
$\PSL(2,\BC)$, we can define the character variety $X_{M,\PSL(2,C)}$ and
its augmented version $\overline{X}_{M,\PSL(2,\BC)}$.

The advantage of the augmented character variety $\overline{X}_{M,\SL(2,\BC)}$
is that given $\ga \in \pi_1(\pt M)$ there is a regular function $e_{\ga}$ 
which sends $[(\rho,z)]$ to the eigenvalue of $\rho(\ga)$ corresponding to $z$.
In contrast, in $X_{M,\SL(2,\BC)}$ only the trace $e_{\ga}+e_{\ga}^{-1}$
of $\rho(\ga)$ is well-defined. Likewise, in 
$\overline{X}_{M,\PSL(2,C)}$ (resp. $X_{M,\SL(2,\BC)}$) only 
$e_{\ga}^2$ (resp. $e_{\ga}^2+e_{\ga}^{-2}$) is defined.

From now on, we will restrict to a geometric component of the $\PSL(2,\BC)$
character variety of $M$ and its lifts. The four character varieties 
associated to $M$ fit in a commutative diagram 
\begin{equation}
\lbl{eq.four}
\xymatrix@-.6pc{ {\overline{X}_{M,\SL(2,\BC)}} \ar[r] \ar[d] & 
{\overline{X}_{M,\PSL(2,\BC)}}\ar[d]  
\\
{X_{M,\SL(2,\BC)}}\ar[r] & {X_{M,\PSL(2,\BC)}}
}
\end{equation}
where the vertical maps are forgetful maps   $[(\rho,z)]\longto [\rho] = \chi_\rho$ and
the horizontal maps are induced by the projection 
$\SL(2,\BC) \longto \PSL(2,\BC)$. 
The vertical maps are generically 2:1 at the geometric components.
The horizontal maps are discussed in \cite[Cor.3.2]{Dn}.

The notation $X_M$ of Section \ref{sec.intro} matches the notation 
$X_M=X_{M,\SL(2,\BC)}$ of this section.

The next lemma describes the coordinates rings of the four versions of
the character variety.

\begin{lemma}
\lbl{lem.ch1}
\begin{enumerate}
  \item   The coordinate ring of $X_{M,\SL(2,\BC)}$ is generated by
$\tr_{g}$ for all $g \in \pi_1(M)$.
  \item   The coordinate ring of $X_{M,\PSL(2,\BC)}$ is generated by
$\tr_{g}^2$ for all $g \in \pi_1(M)$.
  \item   The coordinate ring of $\overline{X}_{M,\SL(2,\BC)}$ is generated by
$\tr_{g}$ for all $g \in \pi_1(M)$ and by $e_{\ga}$ for 
$\ga \in \pi_1(\pt M)$.
\item   The coordinate ring of $\overline{X}_{M,\PSL(2,\BC)}$ is generated by
$\tr_{g}^2$ for all $g \in \pi_1(M)$ and by $e_{\ga}^2$ for 
$\ga \in \pi_1(\pt M)$.
\end{enumerate}
\end{lemma}
The commutative diagram \eqref{eq.four} gives an inclusion of fields
of rational functions:

\begin{equation}
\lbl{eq.fourf}
\xymatrix{
\BQ(\overline{X}_{M,\SL(2,\BC)}) & \ar@{_{(}->}[l] \, \BQ(\overline{X}_{M,\PSL(2,\BC)})
\\
\BQ(X_{M,\SL(2,\BC)}) \ar@{^{(}->}[u] & \ar@{_{(}->}[l] \, \BQ(X_{M,\PSL(2,\BC)}) 
\ar@{^{(}->}[u]
}
\end{equation}
where the vertical field extensions are of degree $2$.

\subsection{The coefficient field of augmented representations}
\lbl{sub.cfield}

A crucial part in our proof of Theorem \ref{thm.1} is the choice of a 
coefficient field of an $\SL(2,\BC)$-representation of $\pi_1(M)$. In this 
section,
we show that the notion of an augmented representation   fits well
with the choice of a coefficient field.

First, let us describe the problem. Given a subgroup $\Ga$ of $\SL(2,\BC)$,
we can define its {\em trace field} $\BQ(\Ga)$ (resp. its 
{\em coefficient field} $E(\Ga)$) by 
$\BQ(\tr(A)\, | \, A \in \Ga)$ (resp. the field 
generated over $\BQ$ by the entries of all elements $A$ of $\Ga$).
The trace field but {\em not} the coefficient field of $\Ga$ is obviously 
invariant under conjugation of $\Ga$
in $\SL(2,\BC)$. In general, it is not possible to choose
a conjugate of $\Ga$ to be a subgroup of $\SL(2,\BQ(\Ga))$. 
The following lemma shows that this is possible after passing to at most
quadratic extension of the trace field.

\begin{lemma}
\lbl{lem.ch2}(\cite[Prop.~3.3]{Ma}\cite[Cor.~3.2.4]{MR})
If $\Ga$ is non-elementary, then $\Ga$ is conjugate to $\SL(2,K)$ where
$K=\BQ(\Ga)(e)$ is an extension of degree $[K:\BQ(\Ga)]\leq 2$, and
$e$ can be chosen to be an eigenvalue of a loxodromic element of $\Ga$.
\end{lemma}
For the definition of a {\em non-elementary} subgroup of $\SL(2,\BC)$
and of a {\em loxodromic} element, see \cite{Ma,MR}. The proof of   Lemma~\ref{lem.ch2}
uses the theory of 4-dimensional {\em quaternion algebras}.

We want to apply Lemma \ref{lem.ch2} to a representation 
$\rho \in R({M,\SL(2,\BC)})$. Recall that the discrete faithful 
representation $\rho_0$
of $\pi_1(M)$ is non-elementary, and that the subset of characters of 
elementary 
representations in the geometric component $X_{M,\SL(2,C)}$ is Zariski closed,
and therefore, finite; see \cite{MR}.

Given a representation $\rho \in R({M,\SL(2,\BC)})$, let 
$\BQ(\rho)$ and $E(\rho)$ denote the {\em trace field} and the 
{\em coefficient field}
of the subgroup $\rho(\pi_1(M)) \subset \SL(2,\BC)$ respectively. Likewise, if 
$(\rho,z) \in 
\overline{R}({M,\SL(2,\BC)})$ is an augmented representation, let 
$\BQ(\rho,z)$ denote the field generated over $\BQ$ by $\tr_g(\rho)$ for 
$g 
\in \pi_1(M)$ and $e_{\ga}$ for $\ga \in \pi_1(\pt M)$. Similarly, we define 
the coefficient field $E(\rho, z)$ associated to the augmented 
representation $(\rho, z)$.

The next lemma follows from Lemma \ref{lem.ch2} and the above discussion.

\begin{lemma}
\lbl{lem.ch3}
\begin{enumerate}
  \item If $\rho \in R({M,\SL(2,\BC)})$ is generic (i.e., non-elementary)
then a conjugate of $\rho$ is defined over a quadratic extension of 
$\BQ(\rho)$. 
  \item If $(\rho,z) \in \overline{R}({M,\SL(2,\BC)})$ is generic 
(i.e., non-elementary) then there exists $N \in \SL(2,\BC)$ so that
$N^{-1}\rho N$ is defined over $E(\rho,z)$.
  \item There exists $N \in \SL(2,\BC)$ such that if $(\rho,z)$ is near the 
discrete faithful representation $(\rho_0,z_0)$, then $N^{-1}(\rho,z) N$ is
defined over $E(\rho,z)$.
\end{enumerate}
\end{lemma}
An alternative version of the above Lemma is possible; see Lemma \ref{lem.ch6}
below.

\subsection{Augmented representations and the shape field}
\lbl{sub.shapefield}

There is an alternative description of the field 
$\BQ(\overline{X}_{M,\PSL(2,\BC)})$ in terms of shape parameters of ideal
triangulations of $M$, which is useful in applications. For completeness, 
we discuss it in this section and the next. 
Let us first describe $\overline{X}_{M,\PSL(2,\BC)}$ in terms 
of {\em pseudo-developing maps}, discussed in detail in \cite[Sec.2.5]{Dn}.
Given $\rho \in R_{M,\PSL(2,\BC)}$, consider a 
$\rho$-equivariant map $\widetilde{M} \longto \BH^3$, where $\BH^3$ 
denotes the 3-dimensional hyperbolic space. Since $\pt M$ is a 2-torus, 
it lifts to a disjoint collection of planes $\BR^2$ in the universal cover 
$\widetilde{M}$.
Let $\overline{M}$ denote the space obtained by cutting $\widetilde{M}$ along 
these
planes, and crushing them into points. Set-theoretically, 
the set $\overline{M}\setminus \widetilde{M}$ of ideal points is in 1-1 
correspondence with the {\em cusps} of $M$ in $\BH^3$, i.e., with the coset 
$\pi_1(M)/\pi_1(\pt M)$. An augmented representation $(\rho,z) \in 
\overline{R}_{M,\PSL(2,\BC)}$ gives a $\pi_1(M)$-equivariant map
$$
D_{(\rho,z)}: \overline{M} \longto \overline{\BH}^3
$$
where $\overline{\BH}^3=\BH^3 \cup \BC\BP^1$ is the compactification of
hyperbolic space by adding a sphere $\BC\BP^1$ at infinity.
Such a map is a pseudo-developing map in \cite[Sec.2.5]{Dn}. An augmented
character $[(\rho,z)] \in \overline{X}_{M,\PSL(2,\BC)}$ does not have a 
unique pseudo-developing map, however every two are homotopic, for example
using a straight line homotopy $t f(x) +(1-t)g(x)$ in $\BH^3$.
Thus, there is a well-defined map:

\begin{equation}
\lbl{eq.pseudo}
\overline{X}_{M,\PSL(2,\BC)}\longto \{\text{Pseudo-developing maps of M,
modulo homotopy rel boundary}\}
\end{equation}
Consider a 4-tuple of distinct points 
$(A,B,C,D) \in (\overline{M}\setminus \widetilde{M})^4$,
and an augmented character $[(\rho,z)] \in \overline{X}_{M,\PSL(2,\BC)}$.
Then, $D_{[(\rho,z)]}$ sends $A,B,C,D$ to four points $A',B',C',D'$
in $\BC\cup\{\infty\}=\BC\BP^1=\pt \BH^3$, and consider their {\em cross-ratio}
$$
cr_{A,B,C,D}[(\rho,z)]=\frac{(A'-D')(B'-C')}{(A'-C')(B'-D')}.
$$
If $A',B',C',D'$ are distinct, then $cr_{A,B,C,D}[(\rho,z)] \in \BC$, else
$cr_{A,B,C,D}[(\rho,z)]$ is undefined. This gives a rational map
$$
cr_{A,B,C,D}: \overline{X}_{M,\PSL(2,\BC)} \longto \BC.
$$

Let $\BQ^{\mathrm{dev}}_M$ denote the field over $\BQ$ generated by 
$cr_{A,B,C,D}$ for all 4-tuples of distinct points of 
$\overline{M}\setminus \widetilde{M}$. 

\begin{lemma}
\lbl{lem.ch4}
We have 
$$
\BQ^{\mathrm{dev}}_M=\BQ(\overline{X}_{M,\PSL(2,\BC)}).
$$
\end{lemma}
The proof will be given in the next section.

\subsection{Ideal triangulations and the gluing equations variety}
\lbl{sub.gluing}

A convenient way to construct the unique hyperbolic structure on $M$, and
its small incomplete hyperbolic deformations 
is using an {\em ideal triangulation} $\calT=(\calT_1,\dots,\calT_s)$ of
$M$ which recovers the complete hyperbolic structure. For a detailed
description of ideal triangulations, see \cite{BP} and also 
\cite[App.10]{BDR-V}. An ideal triangulation $\calT$ which is compatible with
the discrete faithful representation has 
nondegenerate shape parameters $z_j \in \BC\setminus\{0,1\}$ for $j=1,\dots,s$.
Such a triangulation always exists; for example subdivide the canonical 
Epstein-Penner decomposition of $M$ by adding ideal triangles; 
see \cite{EP, BP,PP}.
Once we choose shape parameters for each ideal tetrahedron, one can use
them to give a hyperbolic metric (in general incomplete) in the universal
cover $\widetilde{M}$, once a compatibility condition along the edges
of $\calT$ is satisfied. This compatibility condition defines the so-called
  {\em Gluing Equations variety} $\calG(\calT)$. In the appendix of \cite{BDR-V},
Dunfield describes a map

\begin{equation}
\lbl{eq.Gg}
\calG(\calT) \longto \overline{R}_{M,\PSL(2,\BC)}
\end{equation}
which projects to an injection 

\begin{equation}
\lbl{eq.Ggg}
\calG(\calT) \longto \overline{X}_{M,\PSL(2,\BC)}
\end{equation}
Consider the field $\BQ(z_1,\dots,z_s)$ over $\BQ$ generated by the shape 
parameters 
$z_1,\dots,z_s$.
A priori, $\BQ(z_1,\dots,z_r)$ depends on $M$. The next lemma describes the
fields of rational functions of augmented representations in terms of
the shape field.

\begin{lemma}
\lbl{lem.ch5}
\rm{(a)} We have
\begin{equation}
\lbl{eq.ch5a}
\BQ(\overline{X}_{M,\PSL(2,\BC)})=\BQ(z_1,\dots,z_s)
\end{equation}
and
\begin{equation}
\lbl{eq.ch5b}
\BQ(\overline{X}_{M,\SL(2,\BC)})=\BQ(z_1,\dots,z_s,e_{\l},e_{\mu})
\end{equation}
\rm{(b)} If the image of $(z_1,\dots,z_s) \in \calG(\calT)$ is
$[(\rho,z)] \in \overline{R}_{M,\PSL(2,\BC)}$ under the map \eqref{eq.Gg}, then
the trace field (resp. coefficient field)
of an $\SL(2,\BC)$ lift of $[(\rho,z)]$ is $\BQ(z_1,\dots,z_s)$ 
(resp. $\BQ(z_1,\dots,z_s,e_{\l},e_{\mu})$).
\end{lemma}

\begin{proof}
The shape parameters $z_j$  , for $j=1,\dots,s$, are coordinate functions on
the curve $\calG(\calT)$. In addition, the squares $e_{\l}^2$ and 
$e_{\mu}^2$ of the eigenvalues of a meridian-longitude pair $(\l, \m)$ of 
$\pt M$
are rational functions of the shape parameters $z_j$. Since the map in 
Equation \eqref{eq.Ggg}
is an inclusion of a curve into another, it follows that their fields
of rational functions are equal. This proves Equation \eqref{eq.ch5a}.
Equation \eqref{eq.ch5b} follows from Lemma 
\ref{lem.ch3} and the fact that 
$e_{\l}^2, e_{\mu}^2 \in \BQ(z_1,\dots,z_s)$. This proves part (a). Part (b)
follows from \cite[Cor.3.2.4]{MR}. 
\end{proof}

\begin{proof}(of Lemma \ref{lem.ch4})
It follows by applying verbatim the proof of \cite[Lem.5.5.2]{MR}. 
\end{proof}

Let us end this section with an alternative version of Lemma \ref{lem.ch3}
using shape fields. Recall from \cite[Sec.2]{Dn} that the map in 
Equation \eqref{eq.Gg}
can be defined as follows. Fix a solution $(z_1,\dots,z_s)$
of the Gluing Equations of $\calT$. Lift $\calT$ to an ideal triangulation of 
$\widetilde{M}$, and then map the lift of one ideal tetrahedron to a fixed
ideal tetrahedron of $\BH^3$ of the same shape, and then use 
$\pi_1(M)$-equivariance to send every other ideal tetrahedron to an appropriate
ideal tetrahedron of $\BH^3$, using face-pairings. There is a consistency 
condition, which is satisfied since we are using a solution to the Gluing 
Equations. This defines a developing map and a corresponding 
$\PSL(2,\BC)$-representation $\rho$.
In \cite[App.~10]{BDR-V}, Dunfield describes how to define not
only a representation in $\PSL(2,\BC)$, but also an augmented one
$(\rho,z)$. 

The combinatorial structure of $\calT$ gives a presentation of 
$\Pi = \pi_1(M)$ in terms of {\em face-pairings}: 
\begin{equation}
\lbl{WirtingerP}
\Pi = \left\langle {g_1, \ldots, g_s \;|\; r_1, \ldots, r_{s-1}} \right\rangle.
\end{equation}
Each generator of $\Pi$ is represented by a path in the $1$-skeleton of the 
dual triangulation of $\calT$; see~\cite[Chap.~5]{MR} or \cite[Ch.11]{R}. 
The entries of $\rho(g_j)$, for $j = 1, \ldots, s$, are given by 
face-pairings, and 
are explicit matrices with entries in $\BQ(z_1,\dots,z_s)$; 
see~\cite[Chap.~5]{MR}. The above discussion
proves the following version of Lemma \ref{lem.ch3}.

\begin{lemma}
\lbl{lem.ch6}
\begin{enumerate}
  \item The image of the map in Equation 
\eqref{eq.Gg} is defined over $\BQ(z_1,\dots,z_s)$.
  \item  Generically, a lift of the image of the map in Equation 
\eqref{eq.Gg} to
$\overline{R}({M,\SL(2,\BC)})$ is defined over 
$\BQ(z_1,\dots,z_s,e_{\l},e_{\mu})$.
\end{enumerate}

\end{lemma}

\section{The non-abelian Reidemeister torsion}
\lbl{sec.proofs}

\subsection{An explanation of the rationality of the Reidemeister torsion in 
dimension 3}
\lbl{sub.explanation}

Before we prove the rationality of the torsion stated in Theorem \ref{thm.1},
let us give the main idea which is rather simple, and defer the 
technical details for the next section.

The starting point is a hyperbolic manifold $M$ with one cusp. 
The character variety $\overline{X}_{M,\SL(2,\BC)}$ depends only on $\pi_1(M)$ 
but we view it in a specific birational equivalent way by using a combinatorial
decomposition of $M$ into ideal tetrahedra. Every such
manifold is obtained by a combinatorial face-pairing of a finite collection 
$\calT$
of nondegenerate (but perhaps flat, or negatively oriented)
ideal tetrahedra $\calT_1,\dots, \calT_s$. The hyperbolic shape of a
nondegenerate ideal tetrahedron is determined by a complex number 
$z \in \BC\setminus\{0,1\}$, up to the action of a finite group of order 6.
The discrete faithful representation $\rho_0$ assigns hyperbolic shapes $z_j$
to the tetrahedra $\calT_j$ for $j=1,\dots,s$.   As we already observe, these shapes satisfy the 
so-called Gluing Equations, which is a collection of polynomial equations
in $z_j$ and $1-z_j$ to make the metric match along the edges of the ideal
tetrahedra. The Gluing Equations define a variety $\calG(\calT)$ which of
course depends on $\calT$. When the discrete faithful representation $\rho_0$
slightly deforms in $\rho_t$ (i.e., bends, in the language of Thurston) this
causes the shapes $z_j$ of $\calT_j$ to deform to $z_j(t)$. For small enough
$t$, the new shapes still satisfy the Gluing Equations. Consequently, for 
every $t$, the shapes $z_j(t)$  , for $j=1,\dots,s$, are algebraically dependent, 
and
so is any algebraic function of the shapes.

In the case of the $A$-polynomial, the squares $e_\l (t)^2$ and $e_\m (t)^2$ 
of the
eigenvalues $e_\l (t)$ and $e_\m (t)$ of a meridian-longitude pair of 
$T^2=\pt M$ 
are rational functions in $z_j(t)$ (in fact, monomials in $z_j(t)$ and 
$1-z_j(t)$ with integer exponents), thus
$(e_\l (t),e_\m (t))$ are algebraically dependent. This dependence defines the 
$A$-polynomial.
 
In the case of Reidemeister torsion and Theorem \ref{thm.1}, the torsion 
$\tau_{\mu}(\rho_t)$ of the relevant chain complex is defined over the
field $\BQ(z_1(t),\dots,z_s(t),e_\l (t),e_\m (t))$. In other words all 
matrices 
that compute the torsion (and thus the ratios of their determinants)
have entries in the field  $\BQ(z_1(t),\dots,z_s(t),e_\l (t),e_\m (t))$.

\subsection{Proof of Theorem \ref{thm.1}}
\lbl{sub.thm1}

In this section, we will prove Theorem \ref{thm.1}.
Let $M$ be a one-cusp finite-volume complete hyperbolic $3$-manifold.
Choose an ideal triangulation $\calT=(\calT_1,\dots,\calT_s)$ compatible with 
the discrete faithful representation of $M$ as described above, and let 
$(z_1,\dots,z_s)$
denote the shape parameters of $\calT$. Let $E$ denote the following
field:
$$
\BK=\BQ(z_1,\dots,z_s,e_{\l},e_{\mu})=\BQ(\overline{X}_{M,\SL(2,\BC)})
$$
where the last equality follows from Lemma \ref{lem.ch5}.

Let $J$ denote an open interval in $\BR$ that contains $0$, and 
consider a 1-parameter family $t \in J \mapsto  
z(t)=(z_1(t),\dots,z_s(t)) \in \calG(\calT)$ 
of solutions of the Gluing Equations, 
with image $(\rho_t',z'_t) \in \overline{R}({M,\PSL(2,\BC)})$ under the 
map in Equation \eqref{eq.Gg} and with lift 
$(\rho_t,z_t) \in \overline{R}({M,\SL(2,\BC)})$
where $\rho_0$ is a lift to $\SL(2,\BC)$ of the discrete faithful 
representation of $M$. Fix $\ga$ an essential curve in the boundary torus 
$\pt M$.

We will explain how to define the Reidemeister torsion $\tau_{\ga}(\rho_t)$ 
(for complete definitions 
the reader can refer to Porti's monograph~\cite{Po} and to Turaev's 
book~\cite{TM}), and why 
it coincides with the evaluation of an element of $\BK$ at $\rho_t$.

The 2-skeleton of the combinatorial dual $W$ to $\calT$ is a 2-dimensional
$CW$-complex which is a spine of $M$; see \cite{BP}. Mostow rigidity Theorem 
implies that every homotopy equivalence of $M$ is homotopic to a 
homeomorphism (even to an isometry), and Chapman's theorem concludes 
that every homotopy equivalence of $M$ is simple; \cite{Co}. 
Thus, $W$ is simple homotopy equivalent to $M$, and we can use $W$ to 
compute $\tau_{\ga}(\rho_t)$. 
The ideas of the definition of the non-abelian torsion $\tau_{\ga}(\rho_t)$ 
are the following:

\begin{itemize}
\item[(a)]
Consider the universal cover $\widetilde{W}$ of $W$ and the integral chain 
complex 
$C_*(\widetilde{W}; \ZZ)$ of $\widetilde{W}$ for $*=0,1,2$. The fundamental  
group $\Pi=\pi_1(W)=\pi_1(M)$ acts on $\widetilde{W}$ by covering 
transformations. This action turns the complex $C_*(\widetilde{W}; \ZZ)$ 
into
a $\BZ[\Pi]$-module. The Lie algebra $\sll$ also can be viewed as a 
$\BZ[\Pi]$-module by using the composition $Ad \circ \rho_t$, where $Ad$ 
denotes the adjoint representation  of $\sll$. We let $\sll_{\rho_t}$ denote 
this $\BZ[\Pi]$-module. The {\em twisted chain complex} of $W$ is 
the $\BC$-vector space:
\begin{equation}
\lbl{eq.twistedC}
C_*^{\rho_t} 
= C_*(\widetilde{W}; \ZZ) \otimes_{\ZZ[\Pi]} \sll_{\rho_t}.
\end{equation}
\item[(b)]
The twisted chain complex $C_*^{\rho_t}$ computes the so-called 
{\em twisted homology} of $W$ which is denoted by $H_*^{\rho_t}$. 
The betti numbers of $H_*^{\rho_t}$ are given by (because $\rho_t$ lies in a 
neighborhood of the discrete and faithful representation and thus is   generic, or regular in Porti's language, 
see~\cite[Chap.~3]{Po}): 
$$
\dim_\BC (H_{0}^{\rho_t})=0, \qquad \dim_\BC (H_{1}^{\rho_t})=1,
\qquad \dim_\BC (H_{2}^{\rho_t})=1.
$$
\item[(c)] For $i=1,2$ construct elements $\mathbf{h}^t_{i}$ in 
$C_{i}^{\rho_t}$, 
which project to bases of the twisted homology groups 
$H_{i}^{\rho_t}$.
\item[(d)] Then, the torsion $\tau_{\ga}(\rho_t)$ is an explicit ratio of
determinants; see~\cite{Db} or~\cite[Chap.~3]{Po} and Equation 
\eqref{eq.torsiondef} below.
\end{itemize}
We now give the details of the definition of the non-abelian Reidemeister 
torsion and 
prove Theorem \ref{thm.1}. To clarify the presentation, suppose that
$V_t$ is a 1-parameter family of $\BC$-vector spaces for $t \in J$. 
We will say that $V_t$ is defined over $\BK$ if there exists a 
vector space $V_\BK$ over $\BQ$ such that 
$V_t=(V_\BK \otimes_{\BQ} E(\rho_t,z_t)) \otimes_{\BQ} \BC$ 
for all $t \in J$, where $E(\rho_t,z_t)$
is the coefficient field of $(\rho_t,z_t)$, defined in Section 
\ref{sub.cfield}. Likewise, a 1-parameter family of $\BC$-linear
transformations $T_t \in \Hom_{\BC}(V_t,W_t)$ is defined over $\BK$ 
if $T \in \Hom_{\BQ}(V_\BK,W_\BK) \otimes_{\BQ} \BC$. 
In concrete terms, a 1-parameter family of matrices (resp. vectors) 
is defined over $E$ if its entries (resp. coordinates) lie in $\BK$.

Lemma \ref{lem.ch6} implies the following.

\begin{claim}
\lbl{claim.rho}
The 1-parameter family $(\rho_t,z_t)$ ($t \in J$) is defined over $\BK$.
\end{claim}
Consider the presentation  $\Pi$ in Equation \eqref{WirtingerP} of $\pi_1(M)$
given by face-pairings.
A coordinate description of the chain complex $C_*^{\rho_t}$ is given by
(see \cite{Db})
$$
\lbl{eq.coordchain}
\xymatrix@1{
0 \ar[r] & {\sll^{s-1}}\ar[r]^-{d_2^{\rho_t}}
& {\sll^s}\ar[r]^-{d_1^{\rho_t}} & {\sll}\ar[r] & 0
}
$$
for $*=0,1,2$ where the boundary operators are given by
$$
d_1^{\rho_t}(x_1, \ldots, x_s)  = \sum_{j=1}^s (1 - g_j) \circ x_j,
\text{ and }
d_2^{\rho_t}(x_1, \ldots, x_{s-1}) = {\left( {\sum_{j=1}^{s-1}  
\frac{\partial r_j}{\partial g_k} 
\circ x_j}\right)}_{1 \leqslant k \leqslant s}.
$$
Here $g \circ x = Ad_{\rho_t(g)}(x)$ and $\frac{\partial r_j}{\partial g_k}$
denotes the {\em Fox derivative} of $r_j$ with respect to $g_k$.
The above description of $C_*^{\rho_t}$ and Claim \ref{claim.rho} imply
the following.

\begin{claim}
\lbl{claim.k}
The 1-parameter family $C_*^{\rho_t}$ ($t \in J$) is defined over $\BK$.
\end{claim}

Next, we construct a 1-parameter family of 
basing elements $\mathbf{h}^t_{i}$ for $i=1,2$ and show
that it is defined over $\BK$.
Let $\left\{e^{(i)}_1, \ldots, e^{(i)}_{n_i}\right\}$ be the set of 
$i$-dimensional cells of $W$. We lift them to the universal cover and we 
choose an arbitrary order and an arbitrary orientation for the cells 
$\left\{ {\tilde{e}^{(i)}_1, \ldots, \tilde{e}^{(i)}_{n_i}} \right\}$. 
If $\mathcal{B} = \{\mathbf{a}, \mathbf{b}, \mathbf{c}\}$ is an orthonormal 
basis of $\sll$, then we consider the corresponding  (geometric) basis over 
$\BC$: 
$$
\mathbf{c}^{i}_{\mathcal{B}} = 
\left\{ \tilde{e}^{(i)}_{1} \otimes \mathbf{a}, \tilde{e}^{(i)}_{1} 
\otimes \mathbf{b}, \tilde{e}^{(i)}_{1} \otimes \mathbf{c}, \ldots, 
\tilde{e}^{(i)}_{n_i}\otimes \mathbf{a}, \tilde{e}^{(i)}_{n_i} \otimes 
\mathbf{b}, \tilde{e}^{(i)}_{n_i}\otimes \mathbf{c}\right\}
$$ 
of $C_i^{\rho_t}$. We fix a generator $P^{\rho_t}$ of 
$H_0^{\rho_t}(\partial M) 
\subset C_{0}^{\rho_t}$ 
i.e., $P^{\rho_t} \in \sll$ is such that 
$Ad_{\rho_t(g)}(P^{\rho_t}) = P^{\rho_t}$ for all $g \in \pi_1(\partial M)$. 

\begin{claim}
\lbl{claim.P}
The 1-parameter family $P^{\rho_t}$ ($t \in J$) is defined over $\BK$.
\end{claim}

\begin{proof}
Observe that $P^{\rho_t}$ is a generator of  the intersection 
$$\ker (Ad_{\rho_t(\mu)} - \mathbf{1}) \cap \ker (Ad_{\rho_t(\lambda)} - 
\mathbf{1}).$$ Since this family of vector spaces and linear maps is defined 
over $\BK$ (by Claim \ref{claim.k}), the result follows.
\end{proof}

The canonical inclusion $j\colon \partial M \to M$ induces 
{(see~\cite[Corollary 3.23]{Po})} an isomorphism 
$$
j_*\colon 
H_2^{\rho_t}(\partial M) \to H_2^{\rho_t}(M) \simeq H_2^{\rho_t}(W) 
= \ker d^{\rho_t}_2 
\subset C_2^{\rho_t}.
$$  
Moreover, one can prove that {(see~\cite[Proposition 3.18]{Po})}
$$
H_2^{\rho_t}(\partial M) \cong H_2(\partial M; \ZZ) \otimes \BC.
$$
More precisely, let $\lbrack \! \lbrack \partial M \rbrack \! 
\rbrack \in H_2(\partial M; \ZZ)$ be the fundamental class induced by 
the orientation of $\partial M$, one has 
$H_2^{\rho_t}(\partial M) = \BC
\left[\lbrack \! \lbrack \partial M \rbrack \! \rbrack \otimes 
P^\rho_t\right]$.
The \emph{reference generator} of $H_2^{\rho_t}(M)$ is defined by 

\begin{equation}
\lbl{EQ:Defh2}
\mathbf{h}^t_{2}
= j_*([\lbrack \! \lbrack \partial M \rbrack \! \rbrack \otimes P^{\rho_t}])
\in C_{2}^{\rho_t}.
\end{equation}
Claim~\ref{claim.P} implies that

\begin{claim}
\lbl{claim.h2}
The 1-parameter family $\mathbf{h}^t_{2}$ ($t \in J$) is defined over $\BK$.
\end{claim}
Since $\rho_t$ is near $\rho_0$ and $\ga$ is admissible, 
the inclusion $\iota \colon \ga \longto M$ induces 
(see~\cite[Definition 3.21]{Po}) an \emph{isomorphism} 
$$
\iota^* \colon H^{\rho_t}_1(\ga) \to H^{\rho_t}_1(M)\simeq H_1^{\rho_t}(W) 
= \ker d_1^{\rho_t} / \mathrm{im}\, d_2^{\rho_t}.
$$
The \emph{reference generator} of the first twisted homology group 
$H_1^{\rho_t}(M)$ is defined by
\begin{equation}
\lbl{EQ:Defh1}
\mathbf{h}^t_{1} 
= \iota_*\left(\left[\lbrack \! \lbrack \ga \rbrack \! \rbrack  
\otimes P^\rho_t\right]\right) \in C_{1}^{\rho_t}.
\end{equation}

Claim~\ref{claim.P} implies that:
\begin{claim}
\lbl{claim.h1}
The 1-parameter family $\mathbf{h}^t_{1}$ ($t \in J$) is defined over $\BK$.
\end{claim}
Using the bases described above, the non-abelian Reidemeister torsion of 
the 1-parameter family
$\rho_t$ is defined by:
\begin{equation}
\lbl{eq.torsiondef}
\tau_{\ga}(\rho_t) = \mathrm{Tor}(C_*^{\rho_t}(W; \sll_{\rho_t}), 
\mathbf{c}^*_{\mathcal{B}}, \mathbf{h}_t^{*}) \in \BC^*.
\end{equation}
The torsion $\tau_{\ga}(\rho_t)$ is an invariant of $M$ which is \emph{well 
defined up to a sign}. Moreover, if $\rho_t$ and $\tilde\rho_t$ are two 
1-parameter family of representations which pointwise 
have the same character then 
$\tau_{\ga}(\rho_t) =\tau_{\ga}(\tilde\rho_t)$. Finally, one can observe that 
$\tau_{\ga}(\rho_t)$ 
does not depend on the choice of the invariant vector $P^{\rho_t}$ 
(see~\cite{Db}).

The above discussion implies that 
\begin{claim}
\lbl{claim.tau}
For every essential curve $\ga \in \pt M$,
the 1-parameter family $\tau_{\ga}(\rho_t)$ ($t \in J$) is defined over $\BK$.
\end{claim}
In other words, there exist $\hat
\tau_{\ga} \in \BQ(\overline{X}_{M,\SL(2,\BC)})$ such that for $(\rho_t,z)$
near $(\rho_0,z_0)$ we have $\tau_{\ga}(\rho)=\hat \tau_{\ga}(\rho_t,z)$.
Since the left hand side does not depend on $z$, it follows from Section
\ref{sub.four} that
$\hat \tau_{\ga} \in \BQ(X_{M,\SL(2,\BC)})$. This concludes the
proof of Theorem \ref{thm.1}.
\qed

\subsection{Proof of Theorems  \ref{thm.11} and \ref{thm.2}}
\lbl{sub.thm11}

The proof of Theorem \ref{thm.1} implies that for every admissible curve
$\ga$, the torsion function $\tau_{\ga}$ is the germ of an element of 
$\BQ(\overline{X}_{M,\SL(2,\BC)})$. Theorem \ref{thm.11} follows from 
Theorem \ref{thm.1} and Lemmas \ref{lem.ch3} and  \ref{lem.ch5}. 

Theorem \ref{thm.2} follows from the fact that $\overline{X}_{M,\SL(2,\BC)}$
is an affine complex curve, and its field of rational functions has
transcendence degree $1$. In addition, $\tau_{\ga}$ and $\tr_{\ga}$ are
rational functions on $\overline{X}_{M,\SL(2,\BC)}$.

\subsection{The dependence of the Reidemeister torsion on the admissible 
curve and the $A$-polynomial}
\lbl{sub.Atorsion}

In this section, we discuss the dependence of the non-abelian Reidemeister 
torsion on the admissible
curve. Although this discussion is independent of the proof of Theorem 
\ref{thm.1}, it might be useful in other contexts. Recall that the non-abelian
Reidemeister torsion is defined in terms of the twisted chain complex 
in Equation \eqref{eq.twistedC} which is not acyclic. 
Thus, it requires the choice
of distinguished bases $\mathbf{h}_{i}$ for $i=1,2$. 
Such bases can be chosen once
an {\em admissible curve} $\ga \in \pt M$ is chosen; see ~\cite[Chap. 3]{Po}.
Porti proves that for every homotopically non-trivial curve $\gamma$ in 
$\partial M$, the   discrete and faithful representation $\rho_0$ is $\gamma$-{\em regular}.
The same holds for representations $\rho$ near $\rho_0$. A well-known 
application of Thurston's {\em Hyperbolic Dehn Surgery Theorem} implies 
that $\rho_0 \in X_M$ is a smooth point of $X_M$ and that a neighborhood
$U$ of $\rho_0$ is parametrized by the polynomial function $\tr_{\ga}$;
see for example \cite{NZ} and \cite[Cor.~3.28]{Po}.
Choose a meridian-longitude pair $(\mu,\l)$ in $\pt M$, set
$\tr_{\mu}(\rho_t)=e_\m+e_\m^{-1}$, $\tr_{\l}(\rho)=e_\l+e_\l^{-1}$, and 
consider the 
$A$-polynomial $A_M=A_M(e_\m,e_\l) \in \BZ[e_\m^{\pm 1},e_\l^{\pm 1}]$ of $M$.
For a detailed discussion on the $A$-polynomial of $M$ and its relation
to the various views of the character, see the appendix of 
\cite{BDR-V}.

With the above notation, Porti proves that the dependence of the torsion
on the admissible curve $\ga$ is controlled by the $A$-polynomial. More
precisely,   one has \cite[Cor.~4.9,~Prop.~4.7]{Po}:

\begin{eqnarray}
\lbl{eq:changecurve}
\tau_\mu &=& \tau_\lambda \cdot 
\left( \frac{\tr_\lambda^2 - 4}{\tr_\mu^2 - 4}\right)^{1/2} \cdot 
\frac{\partial \tr_\mu}{\partial \tr_\lambda} \\
\lbl{eq.tauml}
& = & \tau_\lambda \cdot (\mathrm{res}^* \circ (\Delta^*)^{-1})\left( 
\frac{e_\l}{e_\m}\frac{\partial A_M/\partial e_\l}{\partial A_M/\partial e_\m}
\right),
\end{eqnarray}
where $\mathrm{res}^*\colon X_{M, \SL(2, \BC)} \to X_{\partial M, \SL(2, \BC)}$ is the 
restriction-map induced by the usual inclusion 
$\partial M \hookrightarrow M$, and $\Delta^*$ works has follows on the 
trace field 
$$
\Delta^*(\tr_\gamma) = e_\gamma + {e_\gamma}^{-1}.
$$

\subsection*{Acknowledgment}
A first draft of the paper was discussed during a workshop on the Volume 
Conjecture in Strasbourg 2007. The authors wish to thank their organizers, 
S. Baseilhac, F. Costantino and G. Massuyeau for their hospitality,
and M. Heusener, R. Kashaev and W. Neumann for enlightening conversations. 
S.G. wishes to thank N. Dunfield for numerous useful conversations, 
suggestions and for a careful reading of a first draft.

\bibliographystyle{hamsalpha}\bibliography{biblio}
\end{document}